\documentclass[10pt]{amsart}

\usepackage[utf8]{inputenc}
\usepackage[english]{babel}

\usepackage{amsmath}
\usepackage{enumerate}
\usepackage{amsthm,amssymb}
\usepackage{amscd}
\usepackage{color}
\usepackage[a4paper,margin = 1.9cm]{geometry}
\usepackage{hyperref}
\usepackage{comment}
\usepackage{stackrel}

\newtheorem{theorem}{Theorem}
\newtheorem{proposition}[theorem]{Proposition}
\newtheorem{lemma}[theorem]{Lemma}
\newtheorem{corollary}[theorem]{Corollary}

\newtheorem*{assumption}{Assumption}

\numberwithin{theorem}{section}

\numberwithin{equation}{section}

\title{Concentration for Coulomb gases
on compact manifolds}

\author{David García-Zelada}

\begin{document}

\maketitle

\begin{abstract}
We study the non-asymptotic behavior
of a Coulomb gas on a compact Riemannian manifold.
This gas is a symmetric
n-particle Gibbs measure associated
to the two-body interaction energy given by
the Green function.
We encode such a particle system by using
an empirical measure.
Our main result is 
a concentration inequality
in Kantorovich-Wasserstein distance inspired
from the work of Chafaï, Hardy and Maïda on
the Euclidean space. Their proof
involves large deviation techniques
together with an energy-distance comparison
and a regularization procedure
based on the superharmonicity
of the Green function. 
This last ingredient is not available on a manifold.
We solve this problem
by using the heat kernel and its short-time asymptotic
behavior.

\end{abstract}

\section{Introduction}
\label{sec: coulomb}

We shall consider the model of a Coulomb gas 
on a Riemannian manifold introduced in 
\cite[Subsection 4.1]{ldp-david} and study
its non-asymptotic behavior by obtaining a 
concentration inequality for the
empirical measure around its limit.
Let us describe
the model and the main theorem of 
this article.

Let $(M,g)$ be a 
compact Riemannian manifold of volume form $\pi$.
We suppose, for simplicity, 
that $\pi(M) = 1$ so that
$\pi \in \mathcal P(M)$ where 
$\mathcal P(M)$ denotes the space
of probability measures on $M$.
We endow $\mathcal P(M)$ with the 
topology of weak convergence,
i.e. the smallest topology
such that
$\mu \to \int_M f d\mu $ is continuous for every
continuous function $f:M \to \mathbb R$.
Denote by $\Delta: C^\infty(M) \to C^\infty(M)$
the Laplace-Beltrami operator on $(M,g)$. 
We shall say that
	$$G:M \times M \to ( - \infty, \infty]$$ 
is a Green function for $\Delta$ if
it is 
a symmetric continuous function
such that
for every $x \in M$ the function
$G_x: M \to ( - \infty, \infty]$ defined
by $G_x(y) = G(x,y)$ is integrable and
	$$\Delta G_x = -\delta_x + 1$$
in a distributional sense.
It can be proven
that	
such a $G$ is integrable with respect 
to $\pi \otimes \pi$ and that if $f \in C^\infty(M)$
then $\psi: M \to \mathbb R$,
defined by 
	$$\psi(x)=\int_M G(x,y)f(y)d\pi(y),$$
satisfies that	
\begin{equation}
\label{eq: solution}
	\psi \in C^\infty(M)
	\ \ \ \mbox{ and } \ \ \
	\Delta \psi = - f +  \int_M f(y) d\pi(y).
\end{equation}	
In particular, 
$\int_M G_x d\pi$ does not depend on $x \in M$
and the Green function 
is unique up to an additive constant.
See \cite[Chapter 4]	{green-aubin} for a proof
of these results.		
We will denote by 
$G$ the Green function for $\Delta$
such that
\begin{equation}
\label{eq: zerointegral}
	\int_M G_x d\pi = 0
\end{equation}	
for every $x \in M$.

For $x \in M$ the function
$G_x$ may be thought of as 
the potential generated by a particle located
at $x$ in the manifold $M$ and
a `uniform' negative background charge.
The total energy
of a system of $n$ particles of charge
$1/n$ (each particle coming
with a `uniform' negative background charge)
would be $H_n: M^n \to (-\infty,\infty]$ 
defined by
	$$H_n(x_1,...,x_n)
	 = \frac{1}{n^2}\sum_{i < j}	G(x_i, x_j). $$
Take a sequence
$\{\beta_n\}_{n \geq 2}$ of non-negative numbers and consider the sequence of Gibbs probability
measures	 
\begin{equation}
\label{eq: gibbs}
	d\mathbb P_n(x_1,...,x_n) = \frac{1}{Z_n} 
	e^{-\beta_n H_n(x_1,...,x_n)} 
	d \pi^{\otimes_n}(x_1,...,x_n) 
\end{equation}	
where $Z_n$ is such that $\mathbb P_n(M^n) = 1$.
This can be thought of as the Riemannian generalization
of the usual Coulomb gas model described
in \cite{coulomb-serfaty} or
\cite{concentration-chafai}. In 
the particular case of the round
two-dimensional sphere, it is known
(see \cite{spherical-krishnapur}) that
if $\beta_n = 4\pi n^2$
the probability measure $\mathbb P_n$ describes
the eigenvalues of the quotient
of two independent 
$n \times n$ matrices with independent 
Gaussian entries.
Define $H: \mathcal P(M) \to (-\infty,\infty]$
by
	$$H(\mu) =\frac{1}{2} \int_{M \times M} G(x,y)
					 d\mu(x) d\mu(y).$$
This is a convex lower semicontinuous function
(see \cite[Subsection 4.1]{ldp-david}).		
Let $i_n : M^n \to \mathcal P(M) $
be defined by
	$$i_n(x_1,...,x_n) =
	\frac{1}{n} \sum_{i=1}^n \delta_{x_i}.$$
If $\beta_n/n \to \infty$, 
the author in \cite{ldp-david}
tells us that 
$\{{i_n}_*(\mathbb P_n)\}_{n \geq 2}$ satisfy
a large deviation principle with 
speed $\beta_n$ and rate function					
$H - \inf H$. 	 

In particular, if $F$ is a closed set
of $\mathcal P(M)$ we have
	$$\limsup_{n \to \infty} \frac{1}{\beta_n}\log 
	\mathbb P_n(i_n^{-1}(F)) 
	\leq -  \inf_{\mu \in F} 
	\left(H(\mu)- \inf H \right).$$
or, equivalently
\begin{equation}
\label{eq: ldp}	
	\mathbb P_n(i_n^{-1}(F)) \leq
	\exp \left(-\beta_n
	\inf_{\mu \in F} 
	\left(H (\mu)- \inf H \right)
	+ o(\beta_n) \right)	.
\end{equation}	
The aim of this article is to understand
the $o(\beta_n)$ term
for some family of closed sets $F$. Suppose
we choose some metric $d$ in $\mathcal P(M)$
that induces the topology of weak convergence.
The unique minimizer of
$H$ is
$\mu_{\mathrm {eq}} = \pi$ (see 
Theorem \ref{theo: assumption H compact manifold}) 
so a nice family of closed sets
are the sets
	$$F_r = \{\mu \in \mathcal P(M) : \
	d(\mu, \mu_{\mathrm {eq}}) \geq r \} $$
for $r > 0$.	
As $H$ is lower semicontinuous we 
have that
$\inf_{\mu \in F_r} 
\left(H (\mu)- \inf H \right)$
is strictly positive
and the large deviation inequality is not vacuous.	
We would like a simple expression
in terms of $r$ for the leading
term, so instead of using
$\inf_{\mu \in F_r} 
\left(H (\mu)- \inf H \right)$
we will use a simple function of $r$.

Let $d_g$ denote the
Riemannian distance. 
The metric
we shall use on $\mathcal P(M)$ is
$W_1:\mathcal P(M) \times \mathcal P(M)
\to \mathbb [0, \infty)$
defined by
\begin{equation}
\label{eq: distance}
W_1(\mu, \nu) = \inf \left\{ \int_{M \times M}
		d_g(x,y) d\Pi(x,y) :\
		\Pi \mbox{ is a coupling between $\mu$ and
		$\nu$ }
		\right\}
\end{equation}		
and it is known as the Wasserstein or Kantorovich 
metric.
See \cite[Theorem 7.12]{optimal-villani} for a 
proof that it metrizes
the topology of weak convergence.
The main result of this article is the following.

\begin{theorem}[Concentration inequality
for Coulomb gases]
\label{theo: maintheorem}
Let $d$ be the dimension of $M$.
If $d=2$
there exists a constant $C>0$ 
that does not depend on the
sequence $\{\beta_n\}_{n \geq 2}$ such that for 
every $n\geq 2$	and $r \geq 0$
$$	\mathbb P_n \left( W_1(i_n,
	\pi ) \geq r
		\right)
	\leq
	\exp\left({-\beta_n 
		 \frac{r^2}{4}
		+\frac{\beta_n}{8\pi} \frac{\log(n)}{n}
		+
		C \frac{\beta_n}{n} 
		}\right) .$$					
If $d\geq 3$
there exists a constant $C>0$ 
that does not depend on the
sequence $\{\beta_n\}_{n \geq 2}$
such that for 
every $n \geq 2$ and $r \geq 0$				
$$	\mathbb P_n \left( W_1(i_n,
	\pi ) \geq r
		\right)
	\leq
	\exp\left({-\beta_n 
		 \frac{r^2}{4}
		 +
		C \frac{\beta_n}{n^{2/d}} 
		}\right) .$$		
\end{theorem}

In fact, by a slight modification
we will also prove the following 
generalization.

\begin{theorem}[Concentration inequality
for Coulomb gases in a potential]
\label{theo: potential}
Take a twice
continuously differentiable function
$V:M \to \mathbb R$ 
and define
$$H_n(x_1,...,x_n) = \frac{1}{n^2} \sum_{i<j}
G(x_i,x_j) + \frac{1}{n} \sum_{i=1}^n V(x_i) \ \
\mbox{ and }
\ \ 
H(\mu) = \frac{1}{2}\int_{M \times M}
 G(x,y) d\mu(x) d\mu(y)
 + \int_M V(x) d\mu(x).$$
Then $H$ has a unique minimizer
that will be called $\mu_{eq}$.
Suppose $\mathbb P_n$ is defined by
(\ref{eq: gibbs})
and let $d$ be the dimension of $M$.
If $d=2$ 
there exists a constant $C>0$ 
that does not depend on the
sequence $\{\beta_n\}_{n \geq 2}$ such that for 
every $n\geq 2$	and $r \geq 0$
$$	\mathbb P_n \left( W_1(i_n,
\mu_{\mathrm{eq}} ) \geq r
		\right)
	\leq
	\exp\left({-\beta_n 
		 \frac{r^2}{4}
		+\frac{\beta_n}{8\pi} \frac{\log(n)}{n}
		+ n D(\mu_{eq}\| \pi)
		+
		C \frac{\beta_n}{n} 
		}\right) .$$					
If $d\geq 3$
there exists a constant $C>0$ 
that does not depend on the
sequence $\{\beta_n\}_{n \geq 2}$
such that for 
every $n \geq 2$ and $r \geq 0$				
$$	\mathbb P_n \left( W_1(i_n,
\mu_{\mathrm{eq}} ) \geq r
		\right)
	\leq
	\exp\left({-\beta_n 
		 \frac{r^2}{4}
		 + n D(\mu_{eq}\| \pi)
		 +
		C \frac{\beta_n}{n^{2/d}} 
		}\right) .$$		
\end{theorem}

To prove Theorem \ref{theo: maintheorem} 
we follow \cite{concentration-chafai}
in turn inspired
by \cite{concentration-maurel-segala} 
(see also \cite{energy-serfaty}).
We proceed in three steps.
The first part, described
in Section \ref{sec: general}, 
may be used in any measurable
space but it demands an energy-distance
comparison
and a regularization procedure.
The energy-distance comparison
will be explained in Section
\ref{sec: transport} and it may be extended
to include Green functions of some
Laplace-type operators.
The regularization 
by the heat kernel,
in Section \ref{sec: regularization},
will use a short time asymptotic
expansion. It may
apply to more general kind of energies
where a short-time asymptotic expansion
of their heat kernel is known. 
Having acquired all the tools, Section
\ref{sec: final} will complete the proof
of Theorem \ref{theo: maintheorem}
and, by a slight modification,
Theorem \ref{theo: potential}.

\section{Link to an energy-distance
comparison and a regularization procedure}
\label{sec: general}

In this section $M$ may be any measurable space, 
$\pi$ any probability measure on $M$ and
$H_n: M^n \to (-\infty, \infty]$ 
any measurable function bounded from below.
Given $\beta_n > 0$ we define the Gibbs
probability measure by (\ref{eq: gibbs}).
Let $H: \mathcal P(M) \to (-\infty, \infty]$
be any function that has a unique minimizer
$\mu_{\mathrm{eq}}
 \in \mathcal P(M)$. This shall be thought of as
a rate function of some Laplace principle
as in \cite{ldp-david}. 
Consider a metric 
	$$d:\mathcal P(M) \times \mathcal P(M)
	\to [0,\infty)$$ 
on $\mathcal P(M)$
that induces the topology of weak convergence and
define
	$$F_r = \{\mu \in \mathcal P(M): \
	d(\mu, \mu_{\mathrm{eq}}) \geq r \} $$
for $r > 0$.	 We want to understand
a non-asymptotic inequality similar to 
(\ref{eq: ldp}) with an explicit
$o(\beta_n)$ term. 
For this, we consider the following assumption.
\begin{assumption}
We will say that
an increasing convex function
$f:[0,\infty) \to [0, \infty)$ 
satisfies assumption $A$ if,
for all $\mu \in \mathcal P(M)$,
\begin{equation}
\label{eq: A}
\tag{A}
f\left(d(\mu, \mu_{\mathrm{eq}} ) \right)
\leq H(\mu) - H(\mu_{\mathrm{eq}}).
\end{equation}
\end{assumption}
Under assumption \ref{eq: A}, (\ref{eq: ldp})
implies
$$\mathbb P_n(i_n^{-1}(F_r)) \leq
	\exp \left(-\beta_n
	f(r)
	+ o(\beta_n) \right)	.$$ 
We would like to understand the
$o(\beta_n)$ term and
find a bound that does not depend on $r$.
Denote by 
$D(\cdot\| \pi ) 
: \mathcal P(M) \to (-\infty, \infty]$
the 
relative entropy of $\mu$ with respect to $\pi$,
also known as the Kullback–Leibler divergence,
i.e. $D(\mu\|\pi) = \int_M \rho \log \rho \, d\pi$
if $d\mu = \rho\,  d\pi$ and it is
infinity when $\mu$ is not absolutely continuous
with respect to $\pi$.	
The following result is the first part
of the method mentioned in
Section \ref{sec: coulomb}.

\begin{theorem}[General concentration inequality]
\label{theo: general concentration}

Suppose we have two real
numbers
$a_n$ and $b_n$ such that
there exists a measurable function	
$R:M^n \to \mathcal P(M)$
with the following property

$\bullet$ for every $\vec x = (x_1,...,x_n) \in M^n$
we have
	$$H_n(x_1,...,x_n) \geq
		 H(R(\vec x)) - 
		a_n,  
		\ \ \ \mbox{ and }
		\ \ \ 
		d(R(\vec x), i_n(\vec x)) \leq b_n.$$
\\
Let us denote 
$e_n =\int_{M^n} H_n d\mu_{\mathrm{eq}}^{\otimes_n}$
and $e = H(\mu_{\mathrm{eq}})$.
If $f:[0,\infty) \to [0,\infty)$
is an increasing convex function
that satisfies assumption \ref{eq: A}
then	
$$	\mathbb P_n \left( d(i_n,\mu_{\mathrm{eq}} ) 
	\geq r
		\right)
	\leq
	\exp \left(-\beta_n 
		 2 f \left( \frac{r}{2} \right)+
		 n  D(\mu_{\mathrm{eq}}\| \pi) +
		 \beta_n \left(e_n - e \right) +
		\beta_n a_n
				+ \beta_n f(b_n) \right) .$$	
	
\begin{proof}	

We first prove the two following results.
The first lemma is the analogue of
\cite[Lemma 4.1]{concentration-chafai}.

\begin{lemma}[Lower bound of the partition function]
We have the following lower bound.
	$$Z_n \geq 
	\exp \left(- \beta_n e_n
	-n D(\mu_{\mathrm{eq}} \| \pi) \right)		.$$
\begin{proof}
If $d\mu_{\mathrm{eq}} = \rho_{\mathrm{eq}} d\pi$
we have
\begin{align*}
	Z_n &= \int_{M^n} 
	e^{-\beta_n H_n(x_1,...,x_n)} 
	d \pi^{\otimes_n}(x_1,...,x_n) 					\\
	&\geq
	\int_{M^n}
	e^{-\beta_n H_n(x_1,...,x_n)} 
	e^{-\sum_{i=1}^n 1_{\rho_{\mathrm{eq}} > 0}(x_i)
		\log \rho_{\mathrm{eq}}(x_i)}
	d \mu_{\mathrm{eq}}
	^{\otimes_n}(x_1,...,x_n) 				\\
	&\geq
	\int_{M^n}
	e^{-\beta_n H_n(x_1,...,x_n)
	- \sum_{i=1}^n 1_{\rho_{\mathrm{eq}} > 0}(x_i)
		\log \rho_{\mathrm{eq}}(x_i)} 
	d \mu_{\mathrm{eq}}
	^{\otimes_n}(x_1,...,x_n) 				\\
	&\geq
	e^{- \int_{M^n} \left( \beta_n 
	H_n(x_1,...,x_n)
	+\sum_{i=1}^n 1_{\rho_{\mathrm{eq}} > 0}(x_i)
		\log \rho_{\mathrm{eq}}(x_i)  	\right)
		d \mu_{\mathrm{eq}}
		^{\otimes_n}(x_1,...,x_n) 	 } 		\\
	& =
	e^{- \beta_n e_n
	-n D(\mu_{\mathrm{eq}} \| \pi)}				\\	
\end{align*}
where we have used Jensen's inequality to
get the last inequality.
\end{proof}				
				
\end{lemma}

The second lemma will help us in the step
of regularization.

\begin{lemma}[Comparison]
\label{lem: reg vs nonreg}
Take
$\vec x = (x_1,...,x_n) \in M^n$.
If $d(R(\vec x), i_n(\vec x)) \leq b_n$ then
	$$
	f(d(R(\vec x), \mu_{\mathrm{eq}})) \geq 
	2 \, f \left(
	\frac{d(i_n(\vec x), \mu_{\mathrm{eq}})}{2} \right)
	- f(b_n) .$$

\begin{proof}
As 
	$$d(i_n(\vec x),\mu_{\mathrm{eq}}))
	\leq d(i_n(\vec x), R(\vec x)) +
		d(R(\vec x), \mu_{\mathrm{eq}})$$
we have that		
\begin{align*}
	f\left(\frac{d(i_n(\vec x),\mu_{\mathrm{eq}}))}
	{2}\right)
	& \leq f
	\left(\frac{1}{2}d(i_n(\vec x), R(\vec x)) +
		\frac{1}{2}d(R(\vec x), \mu_{\mathrm{eq}})
		\right)		\\
	& \leq
	\frac{1}{2}f\left( d(i_n(\vec x), R(\vec x)) \right)
	+
	\frac{1}{2}f\left( d(R(\vec x), \mu_{\mathrm{eq}})
	 \right)	 \\
	& \leq
	\frac{1}{2}f\left( b_n \right)
	+
	\frac{1}{2}f\left( d(R(\vec x), \mu_{\mathrm{eq}})
	 \right)	
\end{align*}
where we have used that $f$ is increasing
and convex.

\end{proof}
\end{lemma}
\ \\
Now,
define
	$$A_r = \{\vec x \in M^n/\ 
	d(i_n(\vec x), \mu_{\mathrm{eq}} ) \geq r \}.$$
Then
\begin{align*}	
\mathbb P_n(A_r) 
	&=\frac{1}{Z_n} 
	\int_{A_r} e^{-\beta_n H_n(x_1,...,x_n)} 
	d \pi^{\otimes_n}(x_1,...,x_n)					\\ 
													&
	\leq	
	e^{\beta_n e_n
	+ n D(\mu_{\mathrm{eq}} \| \pi)}	
	 \int_{A_r} e^{-\beta_n H(R(\vec x)) + 
			\beta_n f(n)  } 
	d \pi^{\otimes_n}(x_1,...,x_n)					\\
													&
	\leq	
	e^{\beta_n e_n + \beta_n a_n + n  
				D(\mu_{\mathrm{eq}}\| \pi)  }
	 \int_{A_r} e^{-\beta_n H(R(\vec x))} 
	d \pi^{\otimes_n}(x_1,...,x_n)					\\
		&		
	\stackrel{(*)}{\leq	}
	e^{\beta_n \left(e_n - e \right) + \beta_n a_n + n  
				D(\mu_{\mathrm{eq}}\| \pi)  }
	 \int_{A_r} e^{-\beta_n 
	 f(d(R(\vec x),\mu_{\mathrm{eq}} ))	 } 
	d \pi^{\otimes_n}(x_1,...,x_n)					\\	
		&		
	\stackrel{(**)}{\leq	}
	e^{\beta_n \left(e_n - e \right) +
	\beta_n a_n + n  
				D(\mu_{\mathrm{eq}}\| \pi)  }
	 \int_{A_r} e^{-\beta_n 
	 2 f \left(\frac{d(i_n(\vec x, 
	 \mu_{\mathrm{eq}}))}{2
	 }\right)
	 \, +\, \beta_n f(b_n)}
	d \pi^{\otimes_n}(x_1,...,x_n)					\\	
				&		
	\stackrel{(***)}{\leq}	
	e^{\beta_n \left(e_n - e \right) + 
	\beta_n a_n + n  D(\mu_{\mathrm{eq}}\| \pi)  }
	  e^{-\beta_n 2 f \left( \frac{r}{2} \right)
				+ \beta_n f(b_n)} 
					\\	
							&		
	\leq	
	e^{-\beta_n 
		 2 f \left( \frac{r}{2} \right)+
		 n  D(\mu_{\mathrm{eq}}\| \pi) +
		 \beta_n \left(e_n - e \right) +
		\beta_n a_n
				+ \beta_n f(b_n) 
				}
\end{align*}
where in $(*)$ we have used Assumption
\ref{eq: A}, in $(**)$ we have used
Lemma \ref{lem: reg vs nonreg} and
in $(***)$ we have used the monotonicity
of $f$.

\end{proof}

\end{theorem}

In the next section we return to the case of a compact
Riemannian manifold and study a energy-distance
comparison that will imply Assumption 
\ref{eq: A}.

\section{Energy-distance comparison
in compact Riemannian manifolds}
\label{sec: transport}	

We take the notation used in Section
\ref{sec: coulomb}.
The Kantorovich metric $W_1$ defined
in (\ref{eq: distance}) can be written as
	$$W_1(\mu, \nu)
	= \sup \left\{	
	\int_M f d\mu - \int_M f d\nu :
	\ \|f\|_{\mathrm{Lip}} \leq 1 \right\}		$$
where $$\|f\|_{\mathrm{Lip}} = \sup_{x \neq y}
 	\frac{|f(x) - f(y)|}{d_g(x,y)}.$$
This result is known as the 
Kantorovich-Rubinstein theorem
(see \cite[Theorem 1.14]{optimal-villani}).
In the case of a Riemannian manifold,
by a smooth approximation argument 
such as the one in \cite{smooth-azagra},
we can prove that 
	$$W_1(\mu, \nu)
	= \sup \left\{	
	\int_M f d\mu - \int_M f d\nu :
	f \in C^{\infty}(M) \mbox{ and }
	\ \|\nabla f\|_{\infty} \leq 1 \right\}		.$$	
The next theorem gives the energy-distance
comparison required to satisfy Assumption 
\ref{eq: A}. This is the analogue of
\cite[Theorem 1.3]{concentration-maurel-segala}
and \cite[Lemma 3.1]{concentration-chafai}.

\begin{theorem}[Comparison between 
distance and energy]

\label{theo: assumption H compact manifold}
Suppose that $\mu_{\mathrm{eq}} \in \mathcal P(M)$ 
is a probability measure on $M$ such that
$H(\mu_{\mathrm{eq}}) \leq H(\mu)$ for every
$\mu \in \mathcal P(M)$. Then
\begin{equation}
\label{eq: energy-distance}
	\frac{1}{2}W_1(\mu,\mu_{\mathrm{eq}})^2 \leq 
	H(\mu) - H(\mu_{eq})
\end{equation}	
for every $\mu \in \mathcal P(M)$.
This implies, in particular, 
that $H$ has a unique minimizer and
that
Assumption \ref{eq: A}
is satisfied by $f(r) = \frac{r^2}{2}$.
Furthermore, $\mu_{\mathrm{eq}} = \pi$.
\end{theorem}	

Let
$\mathcal F$
be the space of finite signed measures
$\mu$
on $M$ such that
$\int_{M \times M} 
G(x,y) d |\mu|(x) d |\mu|(y) < \infty$.
For convenience we shall define
$\mathcal E: 
\mathcal F  \to (-\infty, \infty]$
by
	$$\mathcal E(\mu) = 
	 \int_{M \times M} G(x,y)
					 d\mu(x) d\mu(y)	$$
so that $\mathcal E(\mu) = 2 H(\mu)$
whenever
$\mu \in \mathcal P(M) 
\cap \mathcal F$. We can also
notice that if
$\mu, \nu \in \mathcal P(M)$ are such 
that $H(\mu)$ and $H(\nu)$ are finite then
$\int G(x,y)d\mu(x) d\nu(y)< \infty$
by the convexity of $H$,
the measure $\mu-\nu$ belongs to 
$ \mathcal F$ and
\begin{equation}
\label{eq: binomiocuadrado}
\mathcal E(\mu-\nu)
=\mathcal E(\mu) + \mathcal E(\nu) 
- 2 \int G(x,y) d\mu(x) d\nu(y).
\end{equation}

We begin by proving the following result that
may be seen as a comparison of distances 
where the `energy distance'
between two probability measures
$\mu,\nu \in \mathcal P(M)$ of finite
energy is defined
as
$\sqrt {\mathcal E(\mu - \nu)}$.	
This is the analogue of
\cite[Theorem 1.1]{concentration-chafai}.
\begin{lemma}[Comparison of distances]
\label{lem: inequality of distances}
Let $\mu, \nu \in \mathcal P(M)$
such that $H(\mu)$ and
$H(\nu)$ are finite. Then
	$$W_1(\mu,\nu) 
	\leq \sqrt {\mathcal E (\mu - \nu)}.$$
\end{lemma}

\begin{proof}
First suppose $\mu$ and $\nu$ differentiable,
i.e. suppose they have a differentiable density
with respect to $\pi$. Define
$U: M \to \mathbb R $ by
	$$U(x) = \int_M G(x,y) \left(d\mu(y) - d \nu(y)\right).$$
Then, as remarked in (\ref{eq: solution}),
we know that $U$ is differentiable and
	$$\Delta U = - \left(\mu - \nu \right).$$
Take $f \in C^\infty(M)$ such that 
$\|\nabla f\|_{\infty} \leq 1$. We can see that
	$$\int_M f \left( d\mu - d\nu \right)
	= -\int_M f \Delta U
	= \int_M \langle \nabla f, \nabla U \rangle d \pi
	\leq
	\| \nabla f\|_2 \|\nabla U \|_2
	\leq 
	\| \nabla f\|_{\infty} \|\nabla U \|_2.$$
We also know that
	$$(\|\nabla U \|_2)^2 = 
	\int_M \langle \nabla U , \nabla U \rangle d \pi
	=
	- \int_M U \Delta U
	=
	\int_M U (d\mu - d\nu)
	= \mathcal E(\mu - \nu).
	 $$
Then,
	$$\int_M f \left( d\mu - d\nu \right)
	 \leq 
	 \| \nabla f\|_{\infty} \|\nabla U \|_2
	 \leq
	  \| \nabla f\|_{\infty} \,
	  \sqrt{ \mathcal E(\mu - \nu)}.$$
This implies that
	$$W_1(\mu,\nu) \leq 
	\sqrt{\mathcal E(\mu - \nu)}.
	  $$
	  
In general, let $\mu, \nu \in \mathcal P(M)$
such that $H(\mu)$ and $H(\nu)$ are finite.
Take two sequences $\{\mu_n\}_{n \in \mathbb N}$
and $ \{\nu_n\}_{n \in \mathbb N}$ 
of differentiable probability measures
that converge to $\mu$ and $\nu$
respectively and such that 
$\mathcal E(\mu_n) \to \mathcal E(\mu)$ and
$\mathcal E(\nu_n) \to \mathcal E(\nu)$
(see \cite{cbeltran} for a proof of their existence)
and proceed by a limit argument.

\end{proof}

The next step to prove
Theorem \ref{theo: assumption H compact manifold}
is a fact that works for general two-body interactions
i.e. $G$ is not necessarily a Green function.
\begin{lemma}[Comparison of energies]
\label{lem: inequality of energies}
Suppose that $\mu_{\mathrm{eq}}$
is a probability measure such that
$H(\mu_{\mathrm{eq}}) \leq H(\mu)$
for every $\mu \in \mathcal P(M)$. Then,
for every $\mu \in \mathcal P(M)$ 
such that
$H(\mu)< \infty$, we have
	$$
	\mathcal E (\mu - \mu_{\mathrm{eq}}) \leq 
	\mathcal E(\mu) - 
	\mathcal E(\mu_{\mathrm{eq}}).$$ 
\end{lemma}	

\begin{proof}
As $H(\mu)$ and 
$H(\mu_{\mathrm{eq}})$
are finite we use
(\ref{eq: binomiocuadrado})
to notice that
the affirmation
	$$
	\mathcal E (\mu - \mu_{\mathrm{eq}}) \leq 
	\mathcal E(\mu) - 
	\mathcal E(\mu_{\mathrm{eq}})$$
is equivalent to
	$$	 \int_{M \times M}G(x,y)d\mu(x) 
	d\mu_{\mathrm{eq}}(y)
	\geq 
	 \mathcal E(\mu_{\mathrm{eq}}).$$
But, if	
	$$	 \int_{M \times M}G(x,y)d\mu(x) 
	d\mu_{\mathrm{eq}}(y)
	<  \mathcal E(\mu_{\mathrm{eq}})$$
were true then, defining
 $\mu_t = (1-t) \mu_{\mathrm{eq}} + t \mu
= \mu_{\mathrm{eq}} + t (\mu - 
\mu_{\mathrm{eq}})$, we would see
that the linear term  of 
$\mathcal E(\mu_t)$ is 
$\int_{M \times M}G(x,y)d\mu(x) 
d\mu_{\mathrm{eq}}(y)
	-  
	\mathcal E(\mu_{\mathrm{eq}}) < 0$. This means that
$\mathcal E(\mu_t) < 
\mathcal E(\mu_{\mathrm{eq}})$ for
$t>0$ small which is a contradiction.
	
\end{proof}

Now we may complete the proof of Theorem
\ref{theo: assumption H compact manifold}.

\begin{proof}[Proof of Theorem 
\ref{theo: assumption H compact manifold}]
Let $\mu_{\mathrm{eq}}$ be a minimizer of $H$
and let $\mu \in \mathcal P(M)$ be a probability
measure on $M$. If $H(\mu)$ is infinite there
is nothing to prove. If it is not,
by Lemma \ref{lem: inequality of distances} and
\ref{lem: inequality of energies}
we conclude (\ref{eq: energy-distance}). 

To prove that $H$ has a unique minimizer
suppose $\tilde \mu_{\mathrm{eq}}$ is another minimizer
and use Inequality
(\ref{eq: energy-distance}) with 
$\mu = \tilde \mu_{\mathrm{eq}}$ to get
$W_1(\tilde \mu_{\mathrm{eq}}, \mu_{eq}) = 0$ 
and, thus,
$\tilde \mu_{\mathrm{eq}} = \mu_{\mathrm{eq}}$. 

Finally, to see
that $\mu_{\mathrm{eq}}=\pi$ we use 
(\ref{eq: zerointegral}).
Then 
$\mathcal E(\mu - \pi) = \mathcal E(\mu) - 
\mathcal E(\pi)$ when
$\mu$ has finite energy. But by Lemma 
\ref{lem: inequality of distances} we know that
$\mathcal E(\mu - \pi) \geq 0$ and then
$\mathcal E(\mu) \geq \mathcal E(\pi)$
for every $\mu \in \mathcal P(M)$ of finite energy.

\end{proof}

In the next section we study a way
to regularize the empirical measures in
the sense of the hypotheses of Theorem
\ref{theo: general concentration}.

\section{Heat kernel regularization of the energy}
	
\label{sec: regularization}	
	
In this section the main tool is the heat kernel
for $\Delta$. A proof of
the following proposition may be found in
\cite[Chapter VI]{eigenvalues-chavel}.

\begin{proposition}[Heat kernel]
\label{proposition: heat}
There exists a unique differentiable function
$p:(0,\infty) \times M \times M \to \mathbb R $
such that
	$$\frac{\partial}{\partial t} p_t(x,y)
	= \Delta_y \, p_t(x,y) \ \ \ \mbox{and}
	\ \ \ \lim_{t \to 0} p_t(x, \cdot) = \delta_x$$
for every $x,y \in M$ and $t>0$.
Such a function will be called the 
heat kernel for $\Delta$.
It is non-negative,
it is mass preserving, i.e.
$$\int_M p_t(x,y) d\pi(y) = 1$$
for every $x \in M$ and $t > 0$,
it is symmetric, i.e.
	$$p_t(x,y) = p_t(y,x)$$ 
for every
$x, y \in M$ and $t > 0$
and it satisfies the semigroup property
i.e.
	$$\int_M p_t(x,y) p_s(y,z) d\pi(y)
	= p_{t+s}(x,z)$$ 
	for every $x,y \in M$ and $t, s > 0$.
Furthermore, 	
	$$\lim_{t \to \infty} p_t(x,y) = 1$$
uniformly on $x$ and $y$.

\end{proposition}

Let $p$
be the heat kernel associated to $\Delta$.
For each point $x \in M$ and $t > 0$ define
the probability measure
$\mu_x^t \in \mathcal P(M)$
 by 
\begin{equation}
\label{eq: mu}
	d\mu_x^t = p_t(x,\cdot) d\pi,
\end{equation}	
 or, more precisely,
 $d\mu_x^t(y) = p_t(x,y) d\pi (y)$.
Then we define
$R_t: M^n \to \mathcal P(M)$ by
	$$R_t(x_1,...,x_n)= 
	\frac{1}{n} \sum_{i=1}^n \mu_{x_i}^t$$
and we want to find $a_n$ and $b_n$ of
the hypotheses of
Theorem \ref{theo: general concentration}
for $R = R_t$.	

We begin by looking for $b_n$.

\subsection{Distance
to the regularized measure}

\begin{proposition}[Distance to
the regularized measure]
\label{prop: bn}
There exists a constant $C>0$ such that
for all $t > 0$ and $\vec x \in M^n$
	$$W_1(R_t(\vec x), i_n(\vec x))
	\leq C \sqrt t.$$

\begin{proof}
The following arguments
are very similar to those in \cite{matching-ledoux} 
and they will be repeated 
for convenience of the reader.
As $W_1: \mathcal P(M) \times \mathcal P(M) \to [0, \infty)$ is the supremum of linear functions,
it is convex. So
	$$W_1(R_t(\vec x), i_n(\vec x))
	\leq \frac{1}{n} \sum_{i=1}^n
	 	W_1(\delta_{x_i}, \mu_{x_i}^t)
	$$
Then, we will try to find a 
constant $C>0$ such that $W_1(\delta_x, \mu_x^t)
\leq C \sqrt t$
for every $x \in M$. As the only coupling
between $\delta_x$ and $\mu_x^t$ is
their product we see that
	$$W_1(\delta_x, \mu_x^t) 
		= \int_M d_g(x,y) d\mu_x^t(y).$$
In fact we will study 
the 2-Kantorovich distance between $\delta_x$
and $\mu_x^t$
\begin{align*}
	D_t(x)
		& = \int_M d_g(x,y)^2 d\mu_x^t(y) 		\\
		& = \int_M d_g(x,y)^2 p_t(x,y) d \pi(y). 
\end{align*}		
If we prove that there exists a constant
$C > 0$ such that for every $x \in M$
\begin{equation}
\label{eq: Dinequality}
	D_t(x) \leq C^2 t
\end{equation}	
we may conclude that
$W_1(\delta_x, \mu_x^t) \leq C \sqrt t$ for every
$x \in M$
by Jensen's inequality.
To get (\ref{eq: Dinequality}) we use
the following lemma which proof
may be found in 
\cite[Section 3.4]{stochastic-hsu} and
\cite[Theorem 3.5.1]{stochastic-hsu}.

\begin{lemma}[Radial process representation]
\label{lem: browniandistance}
Take $x \in M$. Let $X$ be the Markov process
with generator $\Delta$ starting at $x$
(i.e. $X_t = B_{2t}$ where $B$ is 
a Brownian motion on $M$ starting at $x$).
Define $r:M \to [0,\infty)$
by $r(y) = d(x,y)$. Then
$r$ is differentiable $\pi$-almost everywhere
and there exists a non-decreasing process $L$
and a one-dimensional Euclidean 
Brownian motion $\beta$ such that
$$r(X_t) = \beta_{2t} +
 \int_0^{t} \Delta r (X_s) ds - L_t$$
for every $t \geq 0$ where $\Delta r$ is the
$\pi$-almost everywhere defined Laplacian
of $r$.
\end{lemma}

Applying Lemma \ref{lem: browniandistance}
and Itô's formula and then taking
expected values we get
	$$\mathbb E[r(X_t)^2]
	= 2\int_0^t 
	\mathbb E [r(X_s) \Delta r (X_s) ] ds
	- \mathbb E \left[
	2\int_0^t r(X_s) dL_s
	 \right] +
	2t
	\leq
	\int_0^t 
	2\mathbb E [r(X_s) \Delta r (X_s) ] ds
	+ 2 t
	$$
where we are using the notation of Lemma
\ref{lem: browniandistance}. By
\cite[Corollary 3.4.5]{stochastic-hsu}
we know that $r \Delta r$ is bounded in $M$
and as
$D_t(x) = \mathbb E[r(X_t)]$ 
we obtain (\ref{eq: Dinequality})
where the constant $C$ does not depend on $x$.
\end{proof}
\end{proposition}

Now we will look for $a_n$ of 
the hypotheses of Theorem
\ref{theo: general concentration}.

\subsection{
Comparison
between the regularized and 
the non-regularized energy}

\begin{theorem}[Comparison
between the regularized and 
the non-regularized energy]
\label{prop: an}
Let $d$ be the dimension of $M$.
If $d = 2$ there exists a constant $C > 0$
such that, for every $n \geq 2$,
$t \in (0,1]$ and $\vec x \in M^n$,
$$
	H_n(\vec x) 
	\geq  H (R_t(\vec x)) 
 	- 
 	t 
 	+ \frac{1}{8\pi n}\log(t) - \frac{C}{n}.
$$
If $d > 2$ there exists a constant $C > 0$
such that, for every $n \geq 2$,
$t \in (0,1]$ and $\vec x \in M^n$,
$$
	H_n(\vec x) 
	\geq H (R_t(\vec x)) 
 	- 
 	t 
 	- \frac{C}{n t^{\frac{d}{2} - 1}}.
 $$

\end{theorem}

To compare
$H(	R_t(\vec x))$
and $H_n(\vec x)$  we will write,
for $\vec x = (x_1,...,x_n) \in M^n$,
	$$H(	R_t(\vec x)) =
	\frac{1}{n^2} \sum_{i < j}
	\int_{M\times M} G(\alpha,\beta)
	 d\mu_{x_i}^t(\alpha) d\mu_{x_j}^t(\beta) 
	 +
	\frac{1}{2 n^2} \sum_{i=1}^n
	\int_{M\times M} G(\alpha,\beta)
	 d\mu_{x_i}^t(\alpha) d\mu_{x_i}^t(\beta) .$$
Let us define
\begin{align*}
	G_t(x,y) &= \int_{M\times M}G(\alpha,\beta)
	 d\mu_x^t(\alpha) d\mu_y^t(\beta) \\
	 & = \int_{M\times M}G(\alpha,\beta)
	 p_t(x,\alpha) d\pi(\alpha) 
	 p_t(y,\beta) d\pi(\beta).
\end{align*}	 	 
Then we may write
	$$H(R_t(\vec x)) =
	\frac{1}{n^2} \sum_{i<j} G_t(x_i,x_j)
	+\frac{1}{2 n^2} \sum_{i=1}^n G_t(x_i,x_i) .$$
So we want to compare $G_t$ and $G$. 
The idea we shall use is that
if $G$ is the kernel
of the operator $\bar G$ and
$p_t$ is the kernel of the operator $\bar P_t$
then $G_t$ is the kernel of the operator
$\bar P_t \bar G \bar P_t.$
But using the eigenvector decomposition
we can see that
\begin{equation}
\label{eq: integral-green}
\bar G = 
\int_0^{\infty} \left( \bar P_s - e_0 \otimes e_0^* \right) ds
\end{equation}
where $e_0$ is the eigenvector of
eigenvalue $0$, i.e. the constant function 
equal to one. Then
\begin{equation}
\label{eq: integral-regularizedgreen}
\bar P_t \bar G \bar P_t
= \int_0^{\infty} \left( \bar P_{2t + s}
 - e_0 \otimes e_0^* \right) ds
\end{equation} 
where we have used the semigroup property of 
$t \mapsto \bar P_t$, the fact that
$\bar P_t e_0 = e_0$ and $\bar P_t^* = \bar P_t$.

We will prove the previous idea in a somehow 
different
but very related 
way. We begin by proving 
the analogue of (\ref{eq: integral-green}).

\begin{proposition}[Integral representation
of the Green function]
\label{prop: integralgreen}
For every pair of different points
$x,y \in M$ the function
$t \mapsto p_t(x,y)-1$ is integrable.
For every $x \in M$
the negative part of
the function $t \mapsto p_t(x,x)-1$ is integrable.
Moreover, we have the 
following integral representation
of the Green function.
For every $x,y \in M$
	$$G(x,y) = \int_0^\infty 
	\left(p_t(x,y) - 1 \right) dt .$$

\begin{proof}

To prove the integrability of
$t \mapsto p_t(x,y)-1$ we will need to know the 
behavior of $p_t$ for large and short $t$. For the large-time behavior
we have the following result.

\begin{lemma}[Large-time behavior]
\label{lem: large}
There exists $\lambda > 0$
such that for every $T > 0$, $s \geq 0$ 
and $x,y \in M$
\begin{equation}
\label{eq: large-time}
|p_{T+s}(x,y) - 1 | \leq
e^{-\lambda s}
\sqrt {|p_T(x,x) - 1 ||p_T(y,y) - 1 |}.
\end{equation}

\begin{proof}
We follow the same arguments as in the proof of
\cite[Corollary 3.17]{heat-grigoryan}.
By the semigroup property,
the symmetry of $p_t$ and
Cauchy-Schwarz inequality we get
\begin{equation}
\label{eq: cauchy}
\begin{split}
	\left|	p_{T+s}(x,y) - 1 \right|
	&=\left| 
	\int_M \left(p_{\frac{T+s}{2}}(x,z) - 1 \right)
			\left(p_{\frac{T+s}{2}}(z,y) - 1 \right)	 
			d\pi(z)
			\right|						\\
	&\leq 	
	\left\| p_{\frac{T+s}{2}}
	(x,\cdot)-1 \right\|_{L^2}
	\left\| p_{\frac{T+s}{2}}
	(y,\cdot)-1 \right\|_{L^2}		\\
\end{split}
\end{equation}
If $\lambda$ is the first strictly positive
eigenvalue of $-\Delta$ and if $f \in L^2(M)$
we get
	$$\left\|
	\int_M (p_{\frac{s}{2}}(\cdot,z) - 1)f(z) d\pi(z) 
	\right\|_{L^2}
	\leq e^{-\lambda \frac{s}{2}} 
	\left\| \, f - \int f d\pi\right\|_{L^2}.$$
If we choose $f = p_{\frac{T}{2}}(x,\cdot) - 1$
we obtain
\begin{equation}
\label{eq: forx}
	\left\| p_{\frac{T+s}{2}}(x,\cdot)-1
	\right\|_{L^2}
	\leq e^{-\lambda \frac{s}{2}}
		\left\| p_{\frac{T}{2}}
		(x,\cdot)-1\right\|_{L^2}
	=	 e^{-\lambda \frac{s}{2}}
	\sqrt {p_T(x,x)-1}.
\end{equation}	
where we have used the semigroup property
for the last equality.
Similarly, we get
\begin{equation}
\label{eq: fory}
	\left\| p_{\frac{T+s}{2}}(y,\cdot)-1
	\right\|_{L^2}
	\leq 
		 e^{-\lambda \frac{s}{2}}
	\sqrt {p_T(y,y)-1}.
\end{equation}	
By (\ref{eq: cauchy}),
(\ref{eq: forx}) and (\ref{eq: fory})
we may conclude (\ref{eq: large-time}).		
\end{proof}
\end{lemma}

For the short-time behavior, \cite[Theorem 5.3.4]{stochastic-hsu}
implies the following lemma.

\begin{lemma}[Short-time behavior]
\label{lem: short}
Let $d$ be the dimension of $M$. Then there exist
two positive constants $C_1$ and $C_2$ such that
for every $t \in (0,1)$ and $x,y \in M$ we have
$$\frac{C_1}{t^{\frac{d}{2}}} 
e^{-\frac{d(x,y)^2}{4t}}
\leq
p_t(x,y) \leq \frac{C_2}{t^{d- \frac{1}{2}}} 
e^{-\frac{d(x,y)^2}{4t}}.$$
\end{lemma}

The integrability of
$t \mapsto p_t(x,y) - 1$
when $x \neq y$ and the fact that
$\int_0^\infty (p_t(x,x) - 1)dt = \infty$
for every $x \in M$
can be obtained from
Lemma \ref{lem: short} and 
Lemma \ref{lem: large}.

Using Lemma \ref{lem: large}
and the dominated convergence theorem
we obtain the continuity of the function
$(x,y) \mapsto \int_1^\infty (p_t(x,y)-1)\, dt$
at any $(x,y) \in M\times M$.
By the dominated
convergence theorem and Lemma \ref{lem: short}
we obtain the continuity of the function
$(x,y) \mapsto \int_0^1 (p_t(x,y)-1)\, dt$
for $x \neq y$. Using Fatou's lemma we obtain
the continuity of
$(x,y) \mapsto \int_0^1 (p_t(x,y)-1)\, dt$
at $(x,y)$ such that $x = y$.
So, we get that
$(x,y) \mapsto \int_0^\infty (p_t(x,y)-1)\, dt$
is continuous at any $(x,y) \in M\times M$.

Define the continuous
function $K:M \times M \to (-\infty,\infty]$ by
	$$K(x,y) = \int_0^\infty (p_t(x,y) - 1) dt.$$
The following lemma assures
that $K(x,\cdot)$ is integrable
for every $x \in M$.

\begin{lemma}[Global integrability]
For every $x \in M$
$$\int_0^\infty \int_M
|p_t(x,y) - 1| d\pi(y) dt < \infty.$$
\begin{proof}
Take $T>0$.
By Lemma 
\ref{lem: large} we obtain that
$$\int_T^\infty \int_M
|p_t(x,y) - 1| d\pi(y) dt < \infty.$$
On the other hand we have
$$\int_0^T \int_M
|p_t(x,y) - 1| d\pi(y) dt 
\leq \int_0^T \int_M
(p_t(x,y) + 1)d\pi(y) dt
=2T < \infty.$$
\end{proof}
\end{lemma}

Let
$0=\lambda_0 < \lambda_1 \leq \lambda_2,...$ be 
the sequence of eigenvalues of
$- \Delta$ and
$e_0, e_1, e_2,...$ the sequence of
respective eigenfunctions. Then,
for every $\psi \in C^\infty(M)$
$$
\sum_{n=0}^\infty \exp(-\lambda_n t)
 |\langle e_n, \psi \rangle|^2
= \langle \psi, e^{t \Delta } \psi \rangle
=
\int_{M \times M}
 \psi(x) p_t(x, y) \psi(y) d\pi(x) d\pi(y).$$
Equivalently we have
$$
\sum_{n=1}^\infty \exp(-\lambda_n t)
 |\langle e_n, \psi \rangle|^2
=
\int_{M \times M} \psi(x) (p_t(x, y) -1 )
\psi(y) d\pi(x) d\pi(y)$$
and integrating in $t$ from
zero to infinity we obtain
$$
\sum_{n=1}^\infty \frac{1}{\lambda_n}
	|\langle e_n, \psi \rangle|^2
=
\int_{M \times M} \psi(x) K(x,y)
\psi(y) d\pi(x) d\pi(y).$$
By a polarization identity we have
that, for every $\phi, \psi \in C^\infty(M)$,
$$
\sum_{n=1}^\infty \frac{1}{\lambda_n}
	\langle \psi,e_n\rangle 
	\langle e_n, \phi \rangle
=
\int_{M \times M} \psi(x) K(x,y)
\phi(y) d\pi(x) d\pi(y).$$
Taking $\phi = \Delta \alpha$ we get
$$\langle \psi ,\alpha \rangle
- \int_M \psi d\pi \int_M \alpha d\pi
=
\sum_{n=1}^\infty
	\langle \psi,e_n\rangle 
	\langle e_n, \alpha \rangle
=
\int_{M \times M} \psi(x) K(x,y)
\Delta \alpha(y) d\pi(x) d\pi(y).$$
By definition of the Green function 
we know that
$\int_M G(x,y) \Delta \alpha(y) d\pi(y)
= - \alpha(x) + \int_M \alpha d\pi$
and thus
$$\int_{M \times M} \psi(x) G(x,y)
\Delta \alpha(y) d\pi(x) d\pi(y)
=
\int_{M \times M} \psi(x) K(x,y)
\Delta \alpha(y) d\pi(x) d\pi(y).$$
As $\int_M K(x,y) d\pi(y) = 0=
\int_M G(x,y) d\pi(y) $
and by the continuity of $K$ and $G$
we obtain
$G(x,y) = K(x,y)$ for every 
$x, y \in M$.

\end{proof}
\end{proposition}

Now we will state and prove
(\ref{eq: integral-regularizedgreen}).

\begin{proposition}
[Integral representation of the regularized Green
function]
\label{prop: integralregularizedgreen}
For every $t>0$ and
$x,y \in M$
$$	G_t(x,y) = 	
	\int_{2t}^\infty
	\left( p_{s}(x,y) - 1  \right) ds.
$$
\begin{proof}

Take the time (i.e. with respect to t) derivative
(denoted by a dot above the function).
	$$\dot G_t(x,y)
	=
	\int_{M\times M}
	 \dot p_t(x,\alpha) G(\alpha,\beta) p_t(y,\beta) 
	  d\pi(\alpha) 
	 d\pi(\beta)
	 +
	 \int_{M\times M}
	 \dot p_t(x,\alpha) G(\alpha,\beta) 
	 \dot p_t(y,\beta) 
	  d\pi(\alpha) 
	 d\pi(\beta).	 $$	 
We will study the first term of the sum 
(the second being analogous).	
\begin{align*}
	\int_{M\times M}
	 \dot p_t(x,\alpha) G(\alpha,\beta) p_t(y,\beta) 
	  d\pi(\alpha) 
	 d\pi(\beta) 
	 & =
	 \int_{M\times M}
	 \Delta_{\alpha}
	 p_t(x,\alpha) G(\alpha,\beta) p_t(y,\beta) 
	  d\pi(\alpha) 
	 d\pi(\beta) 						\\
	 & =
	 \int_M \left(\int_M
	 \Delta_{\alpha}
	 p_t(x,\alpha) G(\alpha,\beta) 
	 d\pi(\alpha) \right) 
	 p_t(y,\beta)  d\pi\beta) 				\\
	 & =
	 \int_M \left(\int_M
	 p_t(x,\alpha) 	\Delta_{\alpha} G(\alpha,\beta) 
	 d\pi(\alpha) \right) 
	 p_t(y,\beta)  d\pi(\beta) 				\\
	 & = \int_M \left(\int_M
	 p_t(x,\alpha) 	
	 \left(-\delta_{\beta}(\alpha) +  1 \right)
	 d\pi(\alpha) \right) 
	 p_t(y,\beta)  d\pi(\beta) 				\\
	 & = \int_M \left( - p_t(x,\beta) + 1  \right) 
	 p_t(y,\beta)  d\pi(\beta) 				\\	
	 & = - p_{2t}(x,y) + 1\\
\end{align*}
where in the last line we have used
the symmetry and the semigroup
property of $p$.
Using again the symmetry of $p$ we get 
	$$\dot G_t(x,y)=	 - 2 p_{2t}(x,y) + 2, $$
and by integrating we obtain
	$$G_t(x,y) - G_s(x,y)= \int_s^t 
	\left( - 2 p_{2u}(x,y) + 2  \right) du
	=
	\int_{2s}^{2t} 
	\left( - p_{s}(x,y) + 1  \right) ds$$	
for every $0 < s < t <\infty$.
As a consequence of the uniform convergence
of Proposition \ref{proposition: heat}
we can see that
$\mu_x^t$ and 
$\mu_y^t$ defined
in (\ref{eq: mu}) converge
to $\pi$ as $t$ goes to infinity. 
Fix any $T>0$. As
$G_{T+s}(x,y) = 
\int G_T(\alpha,\beta) d\mu_x^s(\alpha)
d\mu_y^s(\beta)
$ for any $s > 0$
and as $G_T$ is continuous
we obtain	
$\lim_{t \to \infty} G_t(x,y) = 
\int_{M \times M} G_T(x,y) d\pi(x) d\pi(y)=0$ 
and then
$$G_t(x,y) = \int_{2t}^\infty (p_s(x,y) - 1) ds.$$
	 
\end{proof}
\end{proposition}

Using
Proposition 
\ref{prop: integralgreen}
and \ref{prop: integralregularizedgreen} we conclude
the following 
inequality.
We can find an analogous
result in \cite[Lemma 5.2]{arakelov-lang}.

\begin{corollary}[Off-diagonal behavior]
\label{cor: off-diagonal}
For every $n \geq 2$, $t >0$ and
$(x_1,...,x_n) \in M^n$
	$$	\sum_{i<j} G(x_i,x_j) 
		\geq \sum_{i<j} G_t(x_i, x_j) -
		t \, n^2.$$
\begin{proof}
As the heat kernel is non-negative,
by Proposition \ref{prop: integralgreen}
and \ref{prop: integralregularizedgreen} we have
that, for every $x, y \in M$,
	$$G(x,y) - G_t(x,y)
	= \int_0^{2t} \left(p_s(x,y) - 1 \right)ds
	\geq -2t .$$
Then, if $(x_1,...,x_n) \in M^n$,
	$$	\sum_{i<j} G(x_i,x_j) 
		\geq \sum_{i<j} G_t(x_i, x_j) -
		t \, n(n-1)
		\geq \sum_{i<j} G_t(x_i, x_j) -
		t \, n^2$$
\end{proof}
\end{corollary}

What is left to understand is
$\sum_{i=1}^n G_t(x_i, x_i)$.
This will be achieved using Proposition
\ref{prop: integralregularizedgreen}
and the short-time asymptotic 
expansion
of the heat kernel.
A particular case
is 
mentioned in \cite[Lemma 5.3]{arakelov-lang}.

\begin{proposition}[Diagonal behavior]
\label{prop: diagonal}
Let $d$ be the dimension of $M$.
If $d = 2$ there exists
a constant $C>0$ such that
for every $t \in (0,1]$ and $x \in M$
$$G_t(x,x) \leq -\frac{1}{4\pi}\log(t) + C.$$
If $d > 2$ there exists
a constant $C>0$ such that
for every $t \in (0,1]$ and $x \in M$
	$$G_t(x,x) \leq 
	\frac{C}{t^{\frac{d}{2} - 1}}.$$
\begin{proof}	
	
By the asymptotic expansion of the heat
kernel (see for instance
\cite[Chapter VI.4]{eigenvalues-chavel}) we have that there exists a constant
$\tilde C>0$ (independent of $x$ and $t$) such that,
for $t \leq 1$,

	$$\left|
	p_t(x,x) - \frac{1}{(4\pi t)^{\frac{d}{2}}}
	\right|
	\leq 
	\tilde C t^{-\frac{d}{2} + 1}.
	$$
Then,
\begin{equation}
\label{eq: asymptotic inequality}
	p_t(x,x) \leq \frac{1}{(4\pi t)^{\frac{d}{2}}}
	+
	\tilde C t^{-\frac{d}{2} + 1}.
\end{equation}	
We know by Proposition \ref{prop: 
integralregularizedgreen}
that 
\begin{align*}
	G_t(x,x)	&= \int_{2t}^\infty 
				(p_s(x,x)-1) ds		\\
			&= \int_{2t}^2 (p_s(x,x) - 1) ds
				+ \int_2^\infty 
				(p_s(x,x) - 1) ds		\\
			&\leq \int_{2t}^2
			\left[\frac{1}{(4\pi s)^{\frac{d}{2}}}
			+
			\tilde C s^{-\frac{d}{2} + 1} - 1
			\right]ds
			+ \int_2^\infty 
			(p_s(x,x) - 1) ds			\\
			&=\int_{2t}^2
			\left[\frac{1}{(4\pi s)^{\frac{d}{2}}}
			+
			\tilde C s^{-\frac{d}{2} + 1}\right]ds
			+ G_2(x,x)
\end{align*}
In the case $d=2$ we obtain that, for
$t \in (0,1]$,
	$$G_t(x,x) \leq -\frac{1}{4\pi} \log(t) + 
	C$$
where $C$ is $2 \tilde C$
plus a bound for $G_2(x,x)$ independent of $x$.	
In the case $d>2$ we use that
$s^{-d/2 + 1} \leq 2s^{-d/2}$
for $s \in (0,1]$ and that $G_2(x,x)$ is
bounded from above to obtain
a constant $C$ such that, for
$t \in (0,1]$,
	$$G_t(x,x) \leq 
	\frac{C}{t^{\frac{d}{2} - 1}}.$$
	
\end{proof}	
\end{proposition}

Knowing the diagonal and off-diagonal
behavior of the regularized 
Green function we can proceed to
prove Proposition \ref{prop: an}.

\begin{proof}
[Proof of Proposition \ref{prop: an}]
Take $\vec x = (x_1,...,x_n) \in M^n$. Then
if $d = 2$ we have
\begin{align*}
	H_n(\vec x) 
	&\geq \frac{1}{n^2} \sum_{i<j} G_t(x_i,x_j)
	- 
 	t 											\\
 	&\geq  
 	\frac{1}{n^2} \sum_{i<j} G_t(x_i,x_j)
 	- 
 	t 
 	+ \frac{1}{2 n^2}\sum_{i=1}^n G_t(x_i,x_i) 
 	+ \frac{1}{8 \pi n}\log(t) - \frac{1}{2n} C	\\
 	&= H (R_t(\vec x)) 
 	- 
 	t 
 	+ \frac{1}{8\pi n}\log(t) - \frac{1}{2n}  C.
\end{align*} 	
where we have used Corollary \ref{cor: off-diagonal}
and Proposition \ref{prop: diagonal}.
If $d > 2$ we proceed in the same way to get
\begin{align*}
	H_n(\vec x) 
	&\geq \frac{1}{n^2} \sum_{i<j} G_t(x_i,x_j)
	- 
 	t 											\\
 	&\geq  
 	\frac{1}{n^2} \sum_{i<j} G_t(x_i,x_j)
 	- 
 	t 
 	+ \frac{1}{2 n^2}\sum_{i=1}^n G_t(x_i,x_i) 
 	+ \frac{ C}{2 n t^{\frac{d}{2} - 1}}	\\
 	&= H (R_t(\vec x)) 
 	- 
 	t 
 	+ \frac{ C}{2 n t^{\frac{d}{2} - 1}}.
\end{align*}	

\end{proof}

Having acquired all the tools to apply
Theorem \ref{theo: general concentration} 
to the case of a Coulomb gas on a compact
Riemannian manifold,
the next section will be devoted
to prove
the main theorem and its almost immediate
extension.

\section{Proof of the concentration inequality
for Coulomb gases}

\label{sec: final}

\begin{proof}[Proof of Theorem \ref{theo: maintheorem}]

First, we notice that
$e_n = \int_M H_n d\mu_{\mathrm{eq}} = \frac{n-1}{n} e
= 0$. To use
Theorem \ref{theo: general concentration} we define
$f(r) = \frac{r^2}{2}$ and
$R = R_{n^{-\frac{2}{d}}}$.
If $d = 2$, by Proposition
\ref{prop: an} and \ref{prop: bn}, we have
that there exists a constant $\tilde C>0$ such that
	$$H_n(\vec x) 
		\geq H ( R(\vec x)) 
 		- \frac{1}{8\pi n}\log(n)
 		- \frac{\tilde C}{n}$$
	$$W_1(R(\vec x), i_n(\vec x))
		\leq \frac{\tilde C}{\sqrt n} $$
for every $\vec x \in M^n$ and
$n\geq 2$ so we can apply
Theorem \ref{theo: general concentration} to obtain
the desired result with $C = \frac{\tilde C^2}{2} + 
\tilde C$.
Similarly, if $d > 2$, by Proposition
\ref{prop: an} and \ref{prop: bn}, we have
that there exists a constant $\tilde C>0$
such that
$$
	H_n(\vec x) 
	\geq H ( R(\vec x)) 
 	- \frac{\tilde C}{n^{\frac{2}{d}}}$$
$$W_1(R(\vec x), i_n(\vec x))
	\leq \frac{\tilde C}{n^{\frac{1}{d}}} $$
for every $\vec x \in M^n$ and
$n\geq 2$ so we we can apply
Theorem \ref{theo: general concentration} to obtain
the desired result with $C = \frac{\tilde C^2}{2} + 
\tilde C$.
\end{proof}

Finally we present the proof of
 \ref{theo: potential}.

\begin{proof}[Proof of Theorem \ref{theo: potential}]

To apply Theorem \ref{theo: general concentration}
we notice that
Assumption \ref{eq: A} is satisfied by
$f(r) = \frac{r^2}{2}$.
Indeed Theorem
\ref{theo: assumption H compact manifold}
is still true for this new $H$ 
except for the caracterization
of the minimizer.
We obtain, in particular, that
$H$ has a unique minimizer.
By a calculation we can see that
$e - e_n = \frac{1}{2n}\int_{M \times M} 
G(x,y) d\mu_{eq}(x)d\mu_{eq}(y)$ which
is of order $\frac{1}{n}$ and will be absorbed
by the constant $C$. To meet
the hypotheses of Theorem
\ref{theo: general concentration}, we need to compare
	$$\frac{1}{n}\sum_{i=1}^n V(x_i) \ \ \ \mbox{ and }
	\ \ \ 
	\frac{1}{n}\sum_{i=1}^n \int_M V d\mu_{x_i}^t.$$
By using the relation
	$$\mathbb E [V(X_t)] = V(x) + 
	\int_0^t \mathbb E[\Delta f (X_s)] ds$$
where $X_t$ is the Markov process with generator
$\Delta$ starting at $x$ we obtain
	$$\left|\mathbb E [V(X_t)] - V(x) \right|
	\leq 
	\hat C t $$
where $\hat C$ is some upper bound to $\Delta V$
and thus
	$$\left|
	\frac{1}{n}\sum_{i=1}^n \int_M V d\mu_{x_i}^t
	-
	\frac{1}{n}\sum_{i=1}^n V(x_i)
	\right| \leq \hat C t.$$
In conclusion, if we choose
$R = R_{n^{-\frac{2}{d}}}$,
there still exists a constant
$C > 0$ such that
	$$
	H_n(\vec x) 
	\geq H ( R(\vec x)) 
 	- \frac{1}{8\pi n}\log(n) - \frac{C}{n}$$
in dimension two and
 	$$
	H_n(\vec x) 
	\geq H ( R(\vec x)) 
	- \frac{ C}{n^{\frac{2}{d}}}$$
in dimension greater than two so that
we can apply Theorem \ref{theo: general concentration}.

\end{proof}

\ \\
\textsc{CEREMADE, UMR CNRS 7534 Université 
Paris-Dauphine, PSL Research university, Place du Maréchal
de Lattre de Tassigny 75016 Paris, France.}
\par
  \textit{E-mail address}: \texttt{garciazelada@ceremade.dauphine.fr}

\end{document}